\documentclass[11pt]{amsart}
\usepackage{amsthm,amsmath,amssymb,amscd,epsfig}

{\end{list}}
{
   \newtheorem{theorem}{Theorem}[section]
   \newtheorem{proposition}[theorem]{Proposition}
   \newtheorem{lemma}[theorem]{Lemma}

   \newtheorem{corollary}[theorem]{Corollary}

}
{\theoremstyle{definition}

   \newtheorem{example}[theorem]{Example}
   \newtheorem{definition}[theorem]{Definition}
      \newtheorem{assumption}[theorem]{Assumption}

}

\newcommand{\QQ}{{\mathbb{Q}}}

\newcommand{\PP}{{\mathbb{P}}}
\newcommand{\ZZ}{{\mathbb{Z}}}
\newcommand{\ZZp}{{\mathbb{Z}_{\geq 0}}}

\newcommand{\bJ}{{\mathbf{J}}}

\newcommand{\bS}{{\mathbf{S}}}
\newcommand{\bX}{{\mathbf{X}}}
\newcommand{\bY}{{\mathbf{Y}}}

\newcommand{\bj}{{\mathbf{j}}}

\newcommand{\cG}{{\mathcal G}}

\newcommand{\cJ}{{\mathcal J}}

\newcommand{\cM}{{\mathcal M}}

\newcommand{\cO}{{\mathcal O}}

\newcommand{\logsch}{\operatorname{log.sch}}
\newcommand{\lct}{\operatorname{lct}}

\newcommand{\Spec}{\operatorname{Spec}}

\newcommand{\Hom}{{\operatorname{Hom}}}

\newcommand{\isom}{\simeq}

\newcommand{\bfA}{{\mathbb A}}

\newcommand{\setmin}{{\smallsetminus}}

\begin{document}
\title{Singularities of log varieties via jet schemes}

\author{Kalle Karu}
%\thanks{The author was partially supported by NSERC}
\address{Department of Mathematics\\ University of British Columbia \\
  1984 Mathematics Road\\
Vancouver, B.C. Canada V6T 1Z2}
\email{karu@math.ubc.ca}

\author{Andrew Staal}
%\thanks{The author was partially supported by NSERC}
\address{Department of Mathematics and Statistics\\ Queen's University \\
  Jeffery Hall, University Avenue\\
Kingston, ON Canada K7L 3N6}
\email{11aps3@queensu.ca}

\thanks{This work was partially supported by NSERC Discovery grant.}

\begin{abstract} We study logarithmic jet schemes of a log scheme and generalize a theorem of M. Musta{\c{t}}{\u{a}} from the case of ordinary jet schemes to the logarithmic case. If $X$ is a normal local complete intersection log variety, then $X$ has canonical singularities if and only if the log jet schemes of $X$ are irreducible. 
\end{abstract}

\maketitle

\section{Introduction}

 M. Musta{\c{t}}{\u{a}} proved in \cite{Mustata1} that a normal, local complete intersection variety $X$ has canonical singularities if and only if the jet schemes $J_m(X)$ are irreducible for all $m$. The goal of this article is to generalize this result to the case of log varieties.
 
 We follow K. Kato's terminology (cf.~ \cite{Kato}) and consider fine and saturated log schemes $\bX=(X,\cM_X)$, where $X$ is a scheme and $\cM_X$ a sheaf of monoids giving the log structure on $X$. We call $\bX$ a log variety if $X$ is a variety. The simplest log structures on $X$ are associated to open embeddings $U\subset X$, with 
\[ \cM_X = \{f\in \cO_X | \text{$f$ is invertible on $U$}\}.\]
Special cases of this include the trivial log structure $U=X$, toric varieties $T\subset Y$, and toroidal embeddings $U_X\subset X$. 

The log jet scheme $\bJ_m(\bX)$ of the log scheme $\bX$ was constructed in full generality by S.~Dutter in \cite{Dutter}. (The special case where the log structure on $X$ is associated to an open embedding $U\subset X$ was also worked out in \cite{Staal}.)  The log jet scheme $\bJ_m(\bX)$ is again a log scheme. Its $\bS$-valued points are the $\bS$-valued log jets in $\bX$:
\[ \Hom_{\logsch} (\bS, \bJ_m(\bX)) \isom \Hom_{\logsch}(\bS\times \bj_m, \bX).\]
Here $\bS$ is any log scheme, $j_m= \Spec k[t]/(t^{m+1})$ with trivial log structure $\cM_{j_m} = \cO^*_{j_m}$. This isomorphism is functorial in $\bS$ and gives the functor of points in $\bJ_m(\bX)$.

We are mostly interested in the underlying scheme  $J_m(\bX)$ of the log jet scheme. The main theorem we prove is the following:

\begin{theorem}\label{thm-main} Let $\bX$ be a normal, local complete intersection log variety. Then $\bX$ has canonical singularities if and only if $J_m(\bX)$ is irreducible for all $m>0$.
\end{theorem}

We need to explain the notion of a normal, complete intersection log variety and when such a variety has canonical singularities. A log variety $\bX$ is a local complete intersection log variety if locally it is a complete intersection in a toric variety $Y$, with the log structure $\cM_X$ being the restriction of the standard log structure $\cM_Y$ on $Y$. We also require that $X$ intersect the torus orbits of $Y$ in proper dimension. The log variety $\bX$ is normal if the underlying variety $X$ is normal.

To define canonical singularities, we need the canonical divisor $K_\bX$ of a log variety $\bX$. This is defined more generally below, but for complete intersections in a toric variety it can be described as follows. Let $T\subset Y$ be a toric variety with its standard log structure. Then $K_\bY := K_Y+D_Y$, where $K_Y$ is the canonical divisor of $Y$ and $D_Y$ is the reduced divisor $Y\setmin T$. The divisor $K_\bY$ is Cartier (in fact, it is trivial).  If $\bX$ is a normal, complete intersection in $\bY$, we define $D_X= D_Y|_X$ and the canonical divisor $K_\bX = K_X+D_X$. By adjunction, this divisor is again Cartier. 

Now consider a morphism of log varieties $f: \tilde{\bX} \to \bX$, such that the underlying morphism of varieties $f: \tilde{X}\to X$ is proper birational. The discrepancy divisor of $f$ is defined as the difference $K_{\tilde{\bX}} - f^*(K_\bX)$. The discrepancy divisor is required to equal $D_{\tilde{X}}-f^*(D_X)$ away from the exceptional locus. We say that the log variety $\bX$ has canonical singularities if the discrepancy divisor is effective for every proper birational log morphism $f$. 

We will show that the log variety $\bX$ is canonical if and only if the pair $(X,D_X)$ is log canonical near $D_X$ and $X\setmin D_X$ is canonical. This implies that to determine if $\bX$ is canonical, it suffices to consider one resolution $f:\tilde{\bX}\to \bX$. Similarly, we show that the irreducibility of the log jet scheme $J_m(\bX)$ can be expressed in terms of irreducibility and dimension of ordinary jet schemes. This reduces the proof of Theorem~\ref{thm-main} to the cases considered in \cite{Mustata1,Mustata2}.

In the theorem we require $\bX$ to be a normal log variety. In particular, all toric varieties are assumed to be normal. This is used in the proof to simplify the treatment of Cartier divisors on $X$. In Kato's theory of log schemes \cite{Kato}, normality does not play a large role. For example, log smooth varieties are usually not  normal. It may be possible to extend the theorem to the non-normal case.

The requirement of the theorem that $\bX$ is a local complete intersection is essential. Similarly to the case of ordinary jets, it implies for example that the jet scheme $J_m(\bX)$ has no components of small dimension.

{\bf Acknowledgements.} We thank Mircea Musta{\c{t}}{\u{a}} for his help with the statement and proof of the key lemma in Section~\ref{sec:mus}.

\section{Log varieties}

We work over an algebraically closed field $k$ of characteristic zero. The scheme-theoretic notation follows \cite{Hartshorne} and for log schemes we refer to \cite{Kato}. All sheaves are taken in the Zariski topology.

\subsection{Basics about log varieties}

A semi-group is a set with a binary, associative operation. A {\em monoid} is a unitary semi-group; it has a unique identity element. We only consider commutative monoids, and every homomorphism of monoids is required to preserve the identity. Let $P = (P,+)$ denote a monoid throughout. Let $X = (X, \cO_X)$ be a scheme. Under multiplication, $\cO_X$ forms a sheaf of (commutative) monoids. A {\em pre-log structure (pre-log str)} on $X$ is the supplemental data of a sheaf $\cM_X$ of (commutative) monoids on $X$ and a homomorphism of sheaves of monoids $\alpha_X : \cM_X \to \cO_X.$ We usually refer only to ``the pre-log str $\cM_X,$'' suppressing $\alpha_X,$ and to $\alpha$ where no ambiguity arises. A morphism of pre-log structures is a morphism of sheaves of monoids $\cM_X\to \cM_X'$ compatible with the maps to $\cO_X$.

A {\em pre-log scheme}  is a scheme endowed with a pre-log str. Pre-log schemes are denoted in the form $\bX = (X,\cM_X).$ A pre-log scheme $\bX$ such that $X$ is a variety is called a {\em pre-log variety.} 

Let $\bX = (X,\cM_X)$ and $\bY = (Y,\cM_Y)$ be pre-log schemes. A morphism (or {\em log morphism}) from $\bX$ to $\bY$ is a pair of data given by a morphism $f: X \to Y$ and a morphism of pre-log structures $f^{\flat}: \cM_Y \to f_* \cM_X$. Such a morphism is denoted by $(f,f^{\flat})$ or simply $f.$

A pre-log str $\cM_X$ on $X$ is called a {\em log structure (log str)} on $X$ if restricting $\alpha$ induces an isomorphism $\alpha^{-1} \cO_X^* \tilde{\to} \cO_X^*.$ A {\em log scheme} is a scheme endowed with a log str. If $\bX$ is both a pre-log variety and a log scheme, we call $\bX$ a {\em log variety.} A morphism of log schemes is a morphism of the underlying pre-log schemes.

Associated to any pre-log str on $X$ is a natural log str $\cM_X^a$  on $X$ formed by the pushout of the homomorphisms $\alpha^{-1} \cO_X^* \to \cO_X^*$ and $\alpha^{-1} \cO_X^* \hookrightarrow \cM_X$: 
\[\begin{CD} 
\cM_X^a @<<< \cM_X \\
@AAA @AAA \\
\cO_X^* @<<< \alpha^{-1} \cO_X^* .
\end{CD}\]
The log str $\cM_X^a$ is called the {\em log str associated to} $\cM_X.$ This log str is universal for homomorphisms from $\cM_X$ to log structures on $X.$

\begin{example} 
The {\em trivial} log str on $X$ is defined by $\cM_X = \cO_X^*$ and the inclusion morphism $\alpha: \cM_X\to \cO_X.$ This log str is initial in the category of all log structures on $X$. For any log scheme $\bY$, a morphism of log schemes $\bY\to \bX$ is equivalent to a morphism of schemes $Y\to X.$

Similarly, the log structure given by $\cM_X = \cO_X$ and the identity morphism $\alpha: \cM_X\to \cO_X$ is final in the category of all log structures on $X$. In this case, for any log scheme $\bY$, a morphism of log schemes $\bX\to \bY$ is equivalent to a morphism of schemes $X\to Y.$
\end{example}

Let $X$ be any scheme, $\bY$ a pre-log scheme, and $f: X \to Y$ a morphism of schemes. The inverse image sheaf $f^{-1}\cM_Y$ together with the composition $f^{-1}\cM_Y \to f^{-1}\cO_Y \to \cO_X$ defines a pre-log str on $X.$ Further, if $\bY$ is a log scheme, the {\em inverse image log str} is defined to be $(f^{-1}\cM_Y)^a$ and is denoted $f^*\cM_Y.$ The inverse image log str is well-behaved; Kato notes that if $\bY$ is a pre-log scheme, then $(f^{-1}\cM_Y)^a \isom f^*\cM_Y^a$ (cf.~ \cite{Kato}). When $f: X\to Y$ is an open or closed embedding, we may denote $f^*\cM_Y$ by $\cM_Y|_X$.

On the other hand, if $\bX$ is a log scheme, Y is any scheme, and $f: X \to Y$ is a morphism of schemes, then the {\em direct image log str} is the fibre product of the morphisms $\cO_Y \to f_* \cO_X$ and $f_* \cM_X \to f_* \cO_X$ (with the obvious structure morphism), denoted by $f_*^{log} \cM_X.$

\begin{example} 
Let $\jmath :U \hookrightarrow X,$ be an open embedding, give $U$ the trivial log str $\cM_U = \cO_U^*$ and let $\cM_X =  \jmath_*^{log} \cM_U$. The local sections of $\cM_X$ are all the regular functions that are invertible on $U.$ We call this $\cM_X$ the {\em log str associated to the open embedding $U\subset X$.}
\end{example}

In practice we only work with certain well-behaved (pre-)log structures. We will need some notions about monoids in order to introduce them. 

Every monoid $P$ contains its {\em  group of units} $P^*;$ the quotient monoid $P/P^*$ is denoted $\overline P.$ Concretely, $\overline P$ is the set of cosets $\{ p+P^* : p\in P \}$ with the natural operation. 

To every monoid $P$ there is associated a group $P^{gp},$ the {\em groupification} of $P,$ defined by the obvious universal property for homomorphisms from $P$ to arbitrary abelian groups. Let $p,p'$ and $n,n'$ denote elements in $P.$ The group $P^{gp}$ is concretely described as the quotient of the product (of monoids) $P \times P$ by the equivalence 
\[ (p,n) \sim (p',n') \iff \exists m \in P : p+n'+m = p'+n+m. \]

If every equality $p+n = p+n'$ in $P$ implies $n = n',$ then $P$ is called {\em integral}. Alternatively, $P$ is integral if the natural homomorphism $P \to P^{gp}$ taking $p \mapsto (p,0)$ is injective. A finitely generated and integral monoid is called {\em fine}. $P$ is called {\em saturated} if for any $p\in P^{gp}$ such that $m p = p+p+\cdots + p$ lies in $P$ for some integer $m>0$,  we have that $p\in P$. 

These notions carry over to sheaves of monoids.
Let the subsheaf of units of a sheaf of monoids $\cM_X$ on $X$ be denoted $\cM_X^*.$ We also denote the sheaf of monoids $\cM_X/\cM_X^*$ by $\overline{\cM}_X,$ and the groupification of  $\cM_X$ by $\cM_X^{gp}.$

A monoid $P$ and a homomorphism of monoids $P \to \cO_X(X)$ determines the constant sheaf of monoids $P_X$ on $X$, the pre-log str $\alpha: P_X \to \cO_X$, and the associated log str $P_X^a\to \cO_X$. A log structure $\cM_X$ is {\em quasi-coherent} if at any point of $X$ there is a monoid $P$ such that locally $\cM_X$ is isomorphic to $P_X^a.$ More precisely, for any point $x \in X,$ the pair of data comprised of an open neighbourhood of $x$, say $\jmath : U \hookrightarrow X,$ and a monoid $P$ with a homomorphism $P \to \cO_X(U)$ together defines a {\em chart of $\cM_X$ at $x$} if $\jmath^* \cM_X \isom P_U^a$ as log structures on $U.$ A {\em quasi-coherent} log scheme carries a quasi-coherent log str, or equivalently, admits a chart at every point. A log scheme is {\em coherent} (resp.~ {\em integral, fine, saturated}) if every point admits a chart $(U,P)$ with finitely generated (resp.~ integral, fine, saturated) $P.$ We are mostly interested in fine and saturated log varieties. If the log structure $\cM_X$ is fine and saturated, then every point $x\in X$ has a chart $(U,P)$ such that the natural map induces an isomorphism $P \isom \overline\cM_{X,x}$ (Lemma~1.6 in \cite{Kato-toric}).

\begin{example} 
Let $X$ be a variety, $D\subset X$ a divisor and $V=X\setmin D$. The log str on $X$ associated to the open embedding $V\subset X$ is fine and saturated. To construct a chart $(U,P)$ at $x\in X$, let $P$ be the monoid of effective Cartier divisors near $x$ that are supported on $D$. This is an integral, finitely generated and saturated monoid. Choose local equations for elements in a basis of $P^{gp}$, defined on an open neighborhood $U$ of $x$.  Monomials in these local equations will give local equations of elements of $P$ and hence a morphism of monoids $P\to \cO_X(U)$. 
\end{example}

\subsection{The monoid algebra}

Given a monoid $P$, the {\em monoid algebra} $k[P]$ is the set of finite formal sums $\{ \sum a_p \chi^p : p \in P, a_p\in k \}$ with formal addition, and multiplication induced by the monoid operation (namely $\chi^p \cdot \chi^{p'} = \chi^{p+p'}$). One may consider the constant sheaf $P_X$ on the space $X = \Spec k[P];$ it yields a natural pre-log str on $X$ arising from the canonical homomorphism $P \hookrightarrow k[P].$ Clearly the log scheme $\bX = (\Spec k[P], P_X^a)$ is quasi-coherent, and is integral, fine, or saturated if and only if the monoid $P$ is integral, fine or saturated. We call $P_X^a$ the {\em standard} log str on the scheme $\Spec k[P]$. When $P$ is a fine and saturated monoid, then $X=\Spec k[P]$ is an affine toric variety (cf.~ \cite{Fulton}). The standard log str on $X$ is the same as the log str associated to the open embedding of the big torus $T\subset X$.

Let $(X,\cM_X)$ be a log scheme.  To give a chart $(U,P)$ at $x\in X$ is the same as to give a morphism of log schemes $(U,\cM_X|_U) \to (Y=\Spec k[P], P_Y^a)$, such that the pullback of $P_Y^a$ is isomorphic to  $\cM_X|_U$. If the log str $\cM_X$ is fine and saturated, we can choose $P$ to be fine and saturated. Moreover, if $X$ is of finite type, we can make $P$ bigger if necessary (by adding units from $\cO^*_U$) and assume that $U\to \Spec k[P]$ is a closed embedding. Thus, a fine and saturated log variety is locally a closed subvariety of a toric variety, with the log str $\cM_X$ being the restriction of the standard log str on the toric variety. 

Generalizing the previous argument slightly, let $Y=\Spec R$ be an affine scheme with log str associated to the morphism $P\to R$ for a monoid $P$. If $\bX$ is any log scheme, then 
a morphism of log schemes $\bX\to \bY$ is determined by a commutative diagram
\[\begin{CD} 
\cM_X(X) @<<< P \\
@VVV @VVV \\
\cO_X(X) @<<< R.
\end{CD}\]
If $R=k[P]$ is the monoid algebra of $P$ and $Y$ has the standard log str, then the diagram is determined solely by the homomorphism of monoids $P \to \cM_X(X).$

\subsection{Stratification of  log varieties}

Let $\bX = (X,\cM_X)$ be a fine and saturated log variety, $\overline{\cM}_{X,x}$ the stalk of the sheaf of monoids $\overline{\cM}_{X}$ at a point $x\in X$, and $(\overline{\cM}_{X,x})^{gp}$ its associated group. This is a finitely generated free abelian group; let $r(x)$ be its rank. The function $r: X\to \ZZp$ divides $X$ into a disjoint union:
\[ X = \cup_i X_i, \quad X_i = r^{-1}(i).\]
Each locally closed set $X_i$ can be given the structure of a locally closed subscheme of $X$ as follows. Near a point $x \in X_i$ the ideal of $X_i$ in $\cO_X$ is generated by $\alpha(\cM_X\setmin \cM_X^\times)$.  

%% new example
\begin{example}
When $Y$ is a toric variety with its standard log str, the stratum $Y_i$ is the union of codimension $i$ torus orbits with reduced scheme structure. A chart $(U,P)$ of a fine and saturated log scheme $\bX$ defines a morphism of schemes $U\to Y=\Spec k[P]$. The strata in $U$ are the inverse images of the strata in $Y$.  
\end{example}

\begin{assumption}
\label{ass-strat} We will assume that $X_i$ has codimension $i$ in $X$ for every $i$.
\end{assumption}

The codimension of $X_i$ cannot be greater than $i$. If the codimension is less than $i$, but $X_0$ is non-empty, then the log jet schemes of $\bX$ are reducible (see Example~\ref{ex-codim} below). With some extra work it may be possible to define the canonical divisor $K_\bX$ for such varieties and prove that they do not have canonical singularities. For simplicity, we will require the assumption.

\subsection{Local complete intersection log varieties}

Recall that a fine and saturated log variety  $\bX=(X,\cM_X)$ is locally a closed subset of a toric variety $Y=\Spec k[P]$, such that the log str $\cM_X$ is the restriction of the standard log str on $Y$.  We say that $\bX$ is a local complete intersection log variety if the closed embedding can be chosen to be a complete intersection. More precisely:

\begin{definition} A log variety $\bX$ is a {\em local complete intersection} log variety if it satisfies Assumption~\ref{ass-strat} and every point $x\in X$ has a chart $(U,P)$ with $P$ a fine and saturated monoid such that $U\to Y=\Spec k[P]$ is a closed embedding whose image is a complete intersection in $Y$.
\end{definition}

Note that a local complete intersection log variety is by definition fine and saturated. Assumption~\ref{ass-strat} means that the closed set $U\subset Y$ intersects all torus orbits in proper dimension.

\begin{example} Let $X$ be a non-singular variety, $D\subset X$ a reduced divisor. Let the log str $\cM_X$ be associated to the open embedding $X\setmin D \subset X$. Then $\bX$ is a local complete intersection log variety. Near $x\in X$, let $f_1,\ldots, f_m$ define the components of $D$. Then locally $X$ is isomorphic to the graph of $f_1 \times \cdots \times f_m$ in $Y= X\times \Spec k[z_1,\ldots,z_m]$. Give $Y$ the log str associated to the embedding $\{z_1\cdots z_m\neq 0\}\subset Y$. Then $X$ is a complete intersection in $Y$ and the log str on $X$ is the restriction of the log str on $Y$. The variety $Y$ is not in general toric. However, since $Y$ is non-singular, locally $Y$ admits an \'etale morphism to a toric variety $Z$, with log str on $Y$ pulled back from $Z$. This means that locally $Y$ can be embedded as a hypersurface in the toric variety $Z\times k^*$ so that the morphism to $Z$ is the projection.
\end{example}

\subsection{Log varieties with canonical singularities}

We say that a log variety $\bX$ is normal if the underlying variety $X$ is normal. 
Consider normal, fine and saturated log varieties satisfying Assumption~\ref{ass-strat}. 

The {\em canonical divisor} of a log variety $\bX$ is 
\[ K_\bX = K_X +D_X,\]
where $K_X$ is the canonical divisor of the variety $X$ and $D_X$ is an effective Weil divisor defined as follows. Let $x\in X$ be the generic point of an irreducible  divisor $D$. If the ideal generated by $\alpha(\cM_X\setmin \cM_X^\times)$ vanishes to order $a$ at $x$, then $D$ appears with coefficient $a$ in $D_X$.

Since $\bX$ satisfies Assumption~\ref{ass-strat}, the support of the divisor $D_X$ is the closure of the stratum $X_1$. The coefficient of $D$ in $D_X$ is equal to the length of the scheme $X_1$ at the generic point of $D$. 

If the log str $\cM_X$ is associated to an open embedding $U\subset X$, then $D_X=X\setmin U$ is a reduced divisor. In particular, if $Y$ is a toric variety with its standard log str, then $D_Y=Y\setmin T$ and hence $K_{\bY} = 0$.  From the adjunction formula we get that $K_\bX$ is a Cartier divisor for any local complete intersection log variety $\bX$.

A morphism of log varieties $f: \bY\to \bX$ is called {\em proper birational} if the underlying morphism of varieties $f: Y\to X$ is proper birational. Assume that $K_\bX$ is Cartier,  $f$ proper birational, and let
\begin{equation} K_\bY = f^* K_\bX +\sum a_i E_i,  \label{eq-ai} \end{equation}
where $E_i$ are irreducible divisors on $Y$; to make the coefficients $a_i$ unique, we require that if $E_i$ is not exceptional, then $a_i$ is the coefficient of $E_i$ in $D_Y-f^*(D_X)$. The log variety $\bX$ has {\em canonical singularities} (or simply, is {\em canonical}) if the {\em total discrepancy divisor} $\sum a_i E_i$ is effective for every proper birational $f: \bY\to \bX$.

Let us recall singularities of pairs (cf.~ \cite{Kollar}). A pair $(X,D)$ consists of a normal variety $X$ and a divisor $D$ in $X$, such that $K_X+D$ is Cartier. For a proper birational morphism $Y\to X$, let
\begin{equation} K_Y = f^* (K_X+D) +\sum b_i E_i,\label{eq-bi} \end{equation}
where  $E_i$ are irreducible divisors on $Y$. One assumes that if $E_i$ is not exceptional, then $b_i$ is the coefficient of $E_i$ in $-f^*(D)$.  The divisor $\sum b_i E_i$ is called the total discrepancy divisor of $f$. The pair $(X,D)$ is log canonical (resp.~ canonical) if the total discrepancy divisor has coefficients $b_i \geq -1$ (resp.~ $b_i\geq 0$). When $D=0$, then the variety $X$ is canonical if $b_i\geq 0$ for all $i$.

\begin{proposition}
A log variety $\bX$ is canonical if and only if the pair $(X,D_X)$ is log canonical and $X\setmin D_X$ is canonical. 
\end{proposition}

\begin{proof}
Let $\bY=(Y,\cM_Y)$ be a log variety. Define $\bY'= (Y,\cM_Y')$, where $\cM_Y'$ is the log str associated to the open embedding $Y\setmin D_Y \subset Y$. Then the image of $\alpha: \cM_Y \to \cO_Y$ lies in the submonoid $\cM_{Y'} \subset \cO_Y$, hence the identity morphism of $Y$ extends uniquely to a log morphism $\bY' \to \bY$.   If $f:\bY\to \bX$ is a proper birational morphism, we can compose this with $\bY'\to \bY$. After such a composition, the discrepancy divisor will not increase. Thus, to check if $\bX$ is canonical, it suffices to consider proper birational morphisms from $\bY$, where the log str $\cM_Y$ is associated to an open embedding $U\subset Y$. By the same argument, it suffices to take $U= Y\setmin f^{-1}(D_X)$. In this case $D_Y = f^{-1}(D_X)$ with reduced structure. Now comparing the discrepancy divisors for $\bX$ and $(X,D_X)$ as in the formulas $(\ref{eq-ai})$ and  $(\ref{eq-bi})$ above, we see that $a_i = b_i+1$ if $E_i$ maps to $D_X$, and $a_i=b_i$ otherwise. This implies the statement of the proposition.
\end{proof}

As in the case of ordinary varieties, having canonical singularities can be checked on one resolution. We say that $f:(Y,D_Y)\to (X,D_X)$ is a log resolution of the pair $(X,D_X)$ if $Y$ is non-singular, $f^{-1}(D_X)$ is a divisor, $D_Y=f^{-1}(D_X)^{red}$ is the divisor with reduced structure, and $D_Y$ together with the exceptional locus of $f$ is a divisor of simple normal crossings.
 
\begin{corollary} Let $\bX$ be a log variety such that $K_\bX$ is Cartier, and $f:(Y,D_Y)\to (X,D_X)$ a log resolution. Then $\bX$ has canonical singularities if and only if the divisor
\[ K_Y+D_Y- f^*(K_X+D_X) \]
is effective. \qed
\end{corollary}

\section{Log jet schemes}

Let $\bj_m$ be the scheme $j_m = \Spec k[t]/(t^{m+1})$ with its trivial log str. For any log scheme $\bS$ the product $\bS \times \bj_m$ exists in the category of schemes with log structures, as described in \cite{Kato}. Concretely, since $j_m$ is a one point space, $S$ and $S\times j_m$ are homeomorphic. Since $j_m$ has the trivial log str, the log str on $S\times j_m$ is the log str associated to the pre-log str $\alpha_S: \cM_S\to \cO_{S\times j_m}$. This is defined by the pushout diagram:
\[\begin{CD} 
\cM_{S\times j_m} @<<< \cM_S \\
@AAA @AAA \\
\cO_{S\times j_m}^* @<<< \alpha_S^{-1} \cO_{S\times j_m}^*.
\end{CD}\]
Note that $\alpha_S^{-1} \cO_{S\times j_m}^* = \cO_S^*$, and 
\[  \cO_{S\times j_m}^* = \cO_S^* \times \cG,\]
where $\cG$ is the sheaf of multiplicative groups 
\[ \cG = \{1+a_1t+ \cdots a_mt^m : a_i \in \cO_S  \}.\]
It follows that $\cM_{S\times j_m}$ is the product of sheaves $\cM_S \times \cG$, with $\alpha: \cM_S \times \cG \to \cO_{S\times j_m}$ mapping
\[ \alpha: (m, 1+a_1t+ \cdots a_mt^m) \mapsto \alpha_S(m)\cdot (1+a_1t+ \cdots a_mt^m).\]

Let $\bX$ be a fine and saturated log scheme. The $m^{\rm th}$ {\em log jet scheme} of $\bX,$ which we denote by $\bJ_m(\bX) = (J_m(\bX), \cM_{J_m(\bX)}),$ represents the functor 
\[ \bS \mapsto \Hom_{\logsch}(\bS\times \bj_m, \bX)\]
 from the category of log schemes to that of sets. The existence of log jet schemes is proved in \cite{Dutter}.

It is well-known that if $X\subset \bfA^n$ is a closed subscheme defined by $l$ equations, then the jet scheme $J_m(X)\subset J_m(\bfA^n) \isom \bfA^{(m+1)n}$ is a closed subscheme defined by $(m+1)l$ equations.  In the next subsections we first recall the local equations for the ordinary jet scheme and then generalize this to the case of log jet schemes.

\subsection{Equations defining the ordinary jet scheme}

Let $Y = \bfA^n = \Spec k[x_1,\ldots,x_n]$. Then 
\[ J_m(Y) = \Spec k[x_i^{(j)}]_{i=1,\ldots,n; j=0,\ldots,m}.\]
 Let $S$ be a scheme, and $R=\cO_S(S)$. A morphism $S \to J_m(Y)$ given by $x_i^{(j)} \mapsto r_i^{(j)} \in R$ corresponds to the $S$-valued jet $S\times j_m \to Y$ given by 
\begin{equation} \label{eq-jet} x_i \mapsto \sum_j r_i^{(j)} \frac{t^j}{j!}.\end{equation}

The coordinate ring of $J_m(Y)$ has a $k$-linear differential operator $d: k[x_i^{(j)}] \to k[x_i^{(j)}]$ that maps $ x_i^{(j)} \mapsto  x_i^{(j+1)}$, $ x_i^{(m)} \mapsto 0$, and satisfies the Leibniz rule. If $f(x_1,\ldots,x_n) \in k[x_1,\ldots,x_n]$, then 
\begin{equation}\label{eq-exp} f(\sum_j x_1^{(j)} \frac{t^j}{j!}, \ldots , \sum_j x_n^{(j)} \frac{t^j}{j!}) = \sum_j d^j f \frac{t^j}{j!}.\end{equation}
It follows from this that if  $X\subset Y$ is the closed subscheme $X=V(f_1,\ldots,f_l)$, then $J_m(X)\subset J_m(Y)$ is also a closed subscheme defined by 
\[ J_m(X) = V(d^j f_i)_{i=1,\ldots,l; j=0,\ldots,m}.\]

For later use we mention that the operator $d$ extends to the ring $k[x_1^{\pm 1}, \ldots, x_n^{\pm 1}][x_i^{(j)}]$ and formula (\ref{eq-exp}) also holds for Laurent polynomials $f\in k[x_1^{\pm 1}, \ldots, x_n^{\pm 1}]$.

\subsection{Equations defining the log jet scheme}

Now let $P\subset P^{gp} \isom \ZZ^n$ be a fine, saturated monoid, and $Y= \Spec k[P]$ a toric variety with its standard log str. Let $e_1,\ldots,e_n \in P$ be a basis for $P^{gp}$, and let $x_1,\ldots,x_n\subset k[P]$ be the corresponding monomials. Note that elements of $k[P]$ are then Laurent polynomials in $x_1,\ldots,x_n$. 

We claim that the log jet scheme of $\bY$ is
\[ J_m(\bY) = \Spec k[P][\frac{x_i^{(j)}}{x_i}]_{i=1,\ldots,n; j=1,\ldots,m},\]
with log str induced by $P\hookrightarrow k[P]$.  A morphism of log schemes $\bS\to \bJ_m(\bY)$  is a commutative diagram
\[ \begin{CD} \cM_S(S) @<<< P \\
@VVV @VVV\\
\cO_S(S) @<<< k[P][\frac{x_i^{(j)}}{x_i}].\\
\end{CD} \]
It is determined by a monoid homomorphism $P\to \cM_S(S)$ and elements $\frac{x_i^{(j)}}{x_i} \mapsto \frac{r_i^{(j)}}{r_i} \in R = \cO_S(S)$. (Note that we view $\frac{x_i^{(j)}}{x_i}$ as variables and $\frac{r_i^{(j)}}{r_i}$ as some elements in $R$, not necessarily quotients of two elements.)

An $\bS$-valued jet in $\bY$ is a morphism $\bS\times \bj_m \to \bY$, which corresponds to a commutative diagram
\[ \begin{CD} \cM_{S\times j_m}(S\times j_m) @<<< P \\
@VVV @VVV\\
R[t]/(t^{m+1}) @<<< k[P].\\
\end{CD} \]
This is determined by a homomorphism of monoids $P\to \cM_{S\times j_m}(S\times j_m)$. The monoid $\cM_{S\times j_m}(S\times j_m)$ is the product of the monoid $\cM_S(S)$ with the group  
\[ G = \{ 1+a_1t + \cdots a_m t^m| a_i\in R \} \subset (R[t]/(t^{m+1}))^*.\]
Hence the diagram above is determined by a homomorphism of monoids $P\to \cM_S(S)$ and a homomorphism of groups $P^{gp} \to G$. Let the group homomorphism be given by
\[ e_i \mapsto 1+ \sum_j \frac{r_i^{(j)}}{r_i} \frac{t^j}{j!}.\]
This gives a bijection between $\bS$-valued points of the log jet scheme and $\bS$-valued jets in $\bY$. One can easily check that this is an isomorphism of functors.

Suppose we have a morphism $\bS\to \bJ_m(\bY)$, given by $\phi: P\to \cM_S(S)$ and $\frac{x_i^{(j)}}{x_i} \mapsto \frac{r_i^{(j)}}{r_i} \in R$. Set $r_i = \alpha(\phi(e_i))\in  R$,  and $r_i^{(j)}  = r_i \cdot \frac{r_i^{(j)}}{r_i}$. Then the ordinary $S$-valued jet into $Y$ underlying the $\bS$-valued jet into $\bY$ is given by a $k$-algebra homomorphism $k[P]\to R[t]/(t^{m+1})$. It maps $x_i\in P$ to 
\[  x_i \mapsto r_i(1+ \sum_{j>0} \frac{r_i^{(j)}}{r_i} \frac{t^j}{j!}) = r_i + \sum_{j>0} r_i^{(j)} \frac{t^j}{j!}.\] 
This formula can be compared with (\ref{eq-jet}) and it is the reason for denoting the coordinates on 
the log jet scheme by $\frac{x_i^{(j)}}{x_i}$.

Later, in the proof of the main theorem, we will need to consider birational morphisms of toric varieties and  the induced morphisms of log jet schemes. Let $P\subset Q$ be two fine and saturated monoids, such that $P^{gp}=Q^{gp}$. The inclusion of monoids gives a morphism of toric varieties $f: Y'=\Spec k[Q] \to Y=\Spec k[P]$. This map is birational. With notation as above, we choose a basis $e_1,\ldots,e_n \in P$ for $P^{gp}$ and denote by $x_1,\ldots,x_n$ the corresponding monomials in both $k[P]$ and $k[Q]$. Then
\[ J_m(\bY') = \Spec k[Q][\frac{x_i^{(j)}}{x_i}], \qquad J_m(\bY) = \Spec k[P][\frac{x_i^{(j)}}{x_i}],\]
and one easily checks that the induced morphism of log jet schemes $J_m(\bY')  \to J_m(\bY)$ takes $\frac{x_i^{(j)}}{x_i}$ to $\frac{x_i^{(j)}}{x_i}$. This result can be restated by saying that the diagram
\[ \begin{CD}  J_m(\bY')  @>>> J_m(\bY)\\
@VVV @VVV\\
Y' @>>> Y\\
\end{CD} \]
is a fibre square. In fact, the diagram is also a fibre square in the category of log schemes when we give each scheme its log structure.

To study the log jet schemes of closed subschemes of $Y$, we need to define the
$k$-linear derivation $d: k[P][\frac{x_i^{(j)}}{x_i}] \to k[P][\frac{x_i^{(j)}}{x_i}]$. Note that we have an inclusion of rings 
\[ k[P][\frac{x_i^{(j)}}{x_i}]  \subset k[x_1^{\pm 1},\ldots, x_n^{\pm 1}] [x_i^{(j)}] .\]
The derivation $d$ on $k[x_1,\ldots, x_n] [x_i^{(j)}]$ defined in the  case of ordinary jet schemes extends to a derivation on the ring $k[x_1^{\pm 1},\ldots, x_n^{\pm 1}] [x_i^{(j)}]$. The subring $k[P][\frac{x_i^{(j)}}{x_i}]$ is invariant by this extension, so we define the desired derivation $d$ as the restriction of the extended $d$. Invariance of the subring under $d$ follows from:
\begin{gather*}
d(\prod_i x_i^{a_i}) =  (\prod_i x_i^{a_i} )(\sum_i a_i \frac{x_i^{(1)}}{x_i}), \\
d( \frac{x_i^{(j)}}{x_i} ) = \frac{x_i^{(j+1)} x_i - x_i^{(1)} x_i^{(j)}}{x_i^2} = \frac{x_i^{(j+1)}}{x_i}- \frac{x_i^{(1)}}{x_i} \frac{x_i^{(j)}}{x_i}.
\end{gather*}

Now let $X\subset Y$ be a closed subscheme defined by $X=V(f_1,\ldots,f_l)$. Let the log str on $X$ be the restriction of the log str on  $Y$. 

\begin{lemma} The log jet scheme $J_m(\bX)$ is a closed subscheme of $J_m(\bY)$, defined by 
\[ J_m(\bX) = V(d^j f_i)_{i=1,\ldots,l; j=0,\ldots,m}.\]
The log str on $J_m(\bX)$ is pulled back from $J_m(\bY)$.
\end{lemma}

\begin{proof} The proof is analogous to the proof in the ordinary jet case.
An $\bS$-valued log jet in $\bX$ is a commutative diagram
\[ \begin{CD} \cM_S(S)\times G @<<< P \\
@VVV @VVV\\
R[t]/(t^{m+1}) @<<< k[P]/(f_1,\ldots,f_l),\\
\end{CD} \]
where $R=\cO_S(S)$ and $G$ is as above. This is determined by a monoid homomorphism $P\to \cM_S(S)\times G$ such that the resulting ring homomorphism $\phi: k[P]\to R[t]/(t^{m+1})$ takes $f_1,\ldots,f_l$ to zero. Recall that $\phi$ maps
\[  x_i \mapsto  r_i + \sum_{j>0} r_i^{(j)} \frac{t^j}{j!}.\]
Since any $f \in k[P]$ is a Laurent polynomial in variables $x_1,\ldots,x_n$, it follows that $\phi$ takes $f$ to
\[ f(x_1^{\pm 1},\ldots, x_n^{\pm 1}) \mapsto \sum_j d^j f \frac{t^j}{j!} |_{x_i=r_i, \frac{x_i^{(j)}}{x_i} = \frac{r_i^{(j)}}{r_i}}.\]
To make sense of a monomial $\chi^p \in k[P] \subset k[x_1^{\pm 1},\ldots, x_n^{\pm 1}] $ evaluated at $x_i=r_i$, we have to define it as the image of $p$ under $P\to \cM_S(S)\to R$.  Therefore, $f$ maps to zero if and only if all $d^j f$ when evaluated at $r_i, \frac{r_i^{(j)}}{r_i}$ are zero. 

The statement of the lemma now follows.
\end{proof}

The lemma implies that if $\bX$ is a local complete intersection log variety, then every component of $J_m(\bX)$ has dimension at least $(m+1)\dim X$.

\subsection{Log jets in terms of ordinary jets}

Let $\bX$ be a fine and saturated log variety. Recall that we have a stratification  of $X$ by locally closed subschemes $X=\cup_j X_j$. Let $J_m(\bX) = \cup_j J_m(\bX)_j$, where $J_m(\bX)_j$ is the inverse image of $X_j$ under the projection $J_m(\bX) \to X$. (Note that $J_m(\bX)_j$ could also stand for the $j$-th stratum in the stratification of $J_m(\bX)$ as a log scheme. However, these two constructions agree because the log str on $J_m(\bX)$ is the  pullback of the log str on  $X$.)

\begin{lemma} $J_m(\bX)_j$ is a locally trivial $\bfA^{mj}$-bundle over $J_m(X_j)$.
\end{lemma}

\begin{proof} Fix a point $x\in X_j$ and a chart $(U,P)$ at $x$, such that $U=\Spec A$ and  $P\isom \overline{\cM}_{X,x}$. We replace $X$ by $U$.

To study morphisms of schemes $S \to J_m(\bX)$, we give $S$ the log str $\cM_S=  \cO_S$ and consider morphisms of log schemes $\bS \to \bJ_m(\bX)$, or equivalently, $\bS$-valued log jets in $\bX$. These correspond to commutative diagrams
\[ \begin{CD} R\times G @<<< P \\
@VVV @VVV\\
R[t]/(t^{m+1}) @<<< A,\\
\end{CD} \]
where $R=\cO_S(S)$ and $G= \{ 1+a_1t + \cdots a_m t^m| a_i\in R \}$ is as above.
The stratum $X_j$ near $x$ is defined by the vanishing of all non-units in $P$. For the jet to lie in $J_m(\bX)_j$, we need that every $p\in P\setmin\{0\}$ maps to an element of the form $\sum_{i>0} r_i t^i$. This means that the monoid homomorphism $P\to R \to R[t]/(t^{m+1})$ maps  $P\setmin\{0\}$ to zero. Then also $P\to R\times G \to R[t]/(t^{m+1})$ maps $P\setmin\{0\}$ to zero.  In other words, the ideal in $A$ defining $X_j$ maps to zero, hence the underlying $S$-valued jet in $X$ is a jet in $X_j$. This gives a morphism $J_m(\bX)_j \to J_m(X_j)$.

Given an $S$-valued jet in $X_j$, the extra data needed to get an $\bS$-valued jet in $\bX$ is a monoid homomorphism $P\to G$, equivalently a group homomorphism $P^{gp} \isom \ZZ^j\to G$. This implies that $J_m(\bX)_j \isom J_m(X_j)\times \bfA^{mj}$.
\end{proof}

\begin{example} \label{ex-codim} Suppose $\bX$ does not satisfy Assumption~\ref{ass-strat} and some $X_j$ has codimension less than $j$. Then for $m$ large we have an inequality
\[ \dim J_m(\bX)_j = \dim J_m(X_j) +mj \geq (m+1)\dim X_j +mj > (m+1)\dim X.\]
Thus, if $X_0$ is non-empty, then the dimension of $J_m(\bX)_j$ is greater than the dimension of $J_m(\bX)$ over some non-empty open locus in $X$, hence the log jet scheme is reducible.
 \end{example}

\section{A Lemma about log canonical pairs} \label{sec:mus}

We consider here pairs of the form $(X, \alpha_1 Z_1+\alpha_2 Z_2+\cdots + \alpha_r Z_r)$, where $Z_1,\ldots, Z_r$ are closed subschemes in a normal $\QQ$-Gorenstein variety $X$ and $\alpha_1,\ldots,\alpha_r$ are non-negative real numbers. Let $f: Y\to X$ be a log resolution of the pair, meaning that $f$ is a proper birational morphism, $Y$ is non-singular, the scheme theoretic inverse images $f^{-1}(Z_i)$ are Cartier divisors, and the union of these inverse images together with the exceptional locus of $f$ forms a divisor of simple normal crossings. Then the pair is called log canonical if the total discrepancy divisor of $f$:
\[ Tot.discr = K_{Y/X} - \sum_i \alpha_i f^{-1} (Z_i) \]
has coefficients $\geq -1$.  The log canonical threshold $\lct(X, \alpha_1 Z_1+\alpha_2 Z_2+\cdots + \alpha_r Z_r)$ is the maximal $\alpha>0$, such that $(X, \alpha(\alpha_1 Z_1+\alpha_2 Z_2+\cdots + \alpha_r Z_r))$ is log canonical.
In particular, the pair is log canonical if and only if its log canonical threshold is $\geq 1$.

The multiplier ideal sheaf of the pair $(X, \alpha_1 Z_1+\alpha_2 Z_2+\cdots + \alpha_r Z_r)$ is, with notation as above,
\[ \cJ(X,  \alpha_1 Z_1+\alpha_2 Z_2+\cdots + \alpha_r Z_r) = f_* \cO_Y(\ulcorner Tot.discr\urcorner).\]
Note that  the multiplier ideal sheaf is equal to $\cO_X$ if and only if the coefficients of the total  discrepancy divisor are $>-1$. The pair is log canonical if and only if  $\cJ(X, (1-\epsilon)( \alpha_1 Z_1+\alpha_2 Z_2+\cdots + \alpha_r Z_r)) = \cO_X$ for any $\epsilon >0$.

For a point $x\in X$, the log canonical threshold of the pair $(X, \alpha_1 Z_1+\cdots + \alpha_r Z_r)$ at $x$ is
\[ \lct_x(X, \alpha_1 Z_1+\cdots + \alpha_r Z_r) = \lct (U, \alpha_1 Z_1|_U+\cdots + \alpha_r Z_r|_U),\]
where $U$ is a sufficiently small open neighborhood of $x$. We say that the pair is log canonical at $x$ if the log canonical threshold at $x$ is $\geq 1$. This is equivalent to the equality of stalks
\[  \cJ(X, (1-\epsilon)( \alpha_1 Z_1+\cdots + \alpha_r Z_r))_x = \cO_{X,x}\]
for any $\epsilon >0$. 

More generally, for a closed subset $C\subset X$, we say that the pair $(X, \alpha_1 Z_1+\cdots + \alpha_r Z_r)$ is log canonical near $C$ if it is log canonical in some open neighborhood of $C$. Equivalently, the pair is log canonical at every point $x\in C$.

  \begin{lemma} \label{lem-mustata} Let $X$ be a non-singular variety and $Z_1,\ldots, Z_r$ irreducible closed subschemes of codimension $c_1,\ldots,c_r$, respectively. Then the pair 
  \[ (X, (c_1+\cdots+c_r)(Z_1\cap\cdots\cap Z_r))\]
  is log canonical if and only if the pair
  \[ (X, c_1 Z_1 + \cdots + c_r Z_r))\]
  is log canonical near $Z_1\cap\cdots\cap Z_r$.
    \end{lemma}
  
\begin{proof} This proof is due to M. Musta{\c{t}}{\u{a}}.

Let $c=c_1+\cdots+c_r$. 

$\Leftarrow.$ Assume that $(X,c_1 Z_1 + \cdots + c_r Z_r )$ is log canonical. Let $a_i$ be the ideal sheaf of $Z_i$, $1\leq 1\leq r$. Since $a_1^{c_1}\cdots a_r^{c_r} \subset (a_1+\cdots+ a_r)^c$, monotonicity of log canonical threshold implies 
\[ \lct(X, c(Z_1\cap\cdots \cap Z_r)) \geq \lct(X,c_1 Z_1 + \cdots + c_r Z_r )\geq 1,\]
and hence $(X, c(Z_1\cap\cdots \cap Z_r))$ is also log canonical.

$\Rightarrow.$ Let $(X, c(Z_1\cap\cdots \cap Z_r))$ be log canonical, but assume by contradiction that $(X,c_1 Z_1 + \cdots + c_r Z_r )$ is not log canonical at some point $x\in Z_1\cap\cdots\cap Z_r$. Consider 
\[ \phi(\lambda_1,\ldots,\lambda_r) = \lct_x(X,\lambda_1 Z_1 + \cdots + \lambda_r Z_r )  \]
as a function of $\lambda_1,\ldots, \lambda_r$ in the domain $\lambda_i\geq 0$, $\sum_i \lambda_i=c$. If some $\lambda_i>c_i$, then $\phi(\lambda_1,\ldots,\lambda_r)<1$ because the pair $(X,\lambda_i Z_i)$ is not log canonical in any neighborhood of $x$. By assumption, $\phi(c_1,\ldots,c_r)<1$, hence $\phi(\lambda_1,\ldots,\lambda_r)<1$ for all $(\lambda_1,\ldots,\lambda_r)$ in the compact domain.  
 Since the log canonical threshold $\phi$ depends continuously on $\lambda_i$, we can find $\epsilon>0$, such that
\[ \phi(\lambda_1,\ldots,\lambda_r) <1-\epsilon\]
for all $\lambda_i\geq 0$, $\sum_i \lambda_i=c$.

A theorem of Jow and Miller \cite{JowMiller} states that the multiplier ideal sheaves satisfy
\[ \cJ(X, \alpha(Z_1\cap\cdots \cap Z_r)) = \sum_{\lambda_1+\cdots +\lambda_r=\alpha} \cJ(X,\lambda_1 Z_1 + \cdots + \lambda_r Z_r ) .\]
Since $(X, c(Z_1\cap\cdots \cap Z_r))$ is log canonical, the left hand side is equal to $\cO_X$ for any $\alpha< c$. Let us fix $\alpha = c -\eta$ for some $\eta>0$. Then on the right hand side for some $\lambda_i\geq 0$, $\lambda_1+\cdots +\lambda_r=c-\eta$,
\[\cJ(X, \lambda_1 Z_1 + \cdots + \lambda_r Z_r)_x = \cO_{X,x},\]
hence
\[ \lct_x(X, \lambda_1 Z_1 + \cdots + \lambda_r Z_r) \geq 1.\]
Scaling this by $\frac{c}{c-\eta}$, we get 
\[ \lct_x(X, \frac{c}{c-\eta} (\lambda_1 Z_1 + \cdots + \lambda_r Z_r )) \geq \frac{c-\eta}{c}.\]
However, $\sum_i \frac{\lambda_i c}{c-\eta} = c$, and hence
\[ \lct_x(X, \frac{c}{c-\eta} (\lambda_1 Z_1 + \cdots + \lambda_r Z_r )) <1-\epsilon.\]
Combining the two inequalities, we get 
\[ \frac{c-\eta}{c} < 1-\epsilon \]
for any $\eta>0$. Taking $\eta<c \epsilon$ gives a contradiction.
\end{proof}

We stated the lemma in the form that will be used below. It can be proved more generally. The variety $X$ needs to be normal, $\QQ$-Gorenstein, $Z_i$ closed subschemes and the numbers $c_i$ must satisfy
\[ c_i \geq \lct (X,Z_i) \] 
near $Z_1\cap\cdots\cap Z_r$. Then the same conclusions hold.

\section{Proof of the main theorem}

We now prove Theorem~\ref{thm-main}. The problem is local, so we assume that $Y=\Spec k[P]$ is a toric variety and $X\subset Y$ is a complete intersection of codimension $c$, and that $X$ is normal and intersects the torus orbits in $Y$ in proper dimension. The log str on $Y$ is the standard one and it induces the log str on $X$. In particular, $D_X = D_Y|_X$. 

We need to prove that $J_m(\bX)$ is irreducible for all $m$ if and only if $(X,D_X)$ is log canonical near $D_X$ and $X\setmin D_X$ is canonical. Note that $X\setmin D_X$ is a complete intersection  in the torus $T=Y\setmin D_Y$ and the log jet scheme $J_m(\bX)$ over   $X\setmin D_X$ is the ordinary jet scheme $J_m(X\setmin D_X)$. By Musta{\c{t}}{\u{a}}'s theorem for local complete intersection varieties \cite{Mustata1}, $X\setmin D_X$ has canonical singularities if and only if $J_m(X\setmin D_X)$ is irreducible for all $m$. We will assume these equivalent condiditons in the following. Thus, we need to show that $J_m(\bX)$ has no components lying over $D_X$ if and only if $(X,D_X)$ is log canonical.

\subsection{The case of non-singular $Y$}

Assume that $Y$ is a non-singular toric variety. For later use we will relax the condition of $X$ being normal. We consider complete intersections $X$ that are non-singular in codimension $1$, but could be non-normal at some points in $D_X$. The notion of $(X,D_X)$ being log canonical is the same as before.

\begin{lemma} \label{lem-reduc} The pair $(X,D_X)$ is log canonical if and only if the pair $(Y, cX+D_Y)$ is log canonical.  
 \end{lemma}
 
 \begin{proof} The proof is the same as the corresponding reduction in \cite{Mustata1}.
 Let $f: \tilde{Y}\to Y $ be the blowup of $Y$ along $X$, with $F$ the exceptional divisor. Then $F$ is a locally trivial $\PP^{c-1}$-bundle over $X$. Hence $(X,D_X)$ is log canonical if and only if $(F,f^*(D_X))$ is log canonical.  Notice that $F$ is also non-singular in codimension $1$.
 
 The log canonicity of $(Y,cX+D_Y)$ can be checked on a resolution that factors through $\tilde{Y}$. The blowup $\tilde{Y}$ is a local complete intersection, hence Gorenstein. Since $F\subset \tilde{Y}$ is non-singular in codimension $1$, so is $\tilde{Y}$. Then from $K_{\tilde{Y}/Y} = (c-1)F$ it follows that  $(Y,cX+D_Y)$ is log canonical if and only if the pair $(\tilde{Y}, F+ f^* D_Y)$ is log canonical. 
 
 Applying inverse of adjunction \cite{Kawakita}, the pair $(F,f^*D_X) = (F, f^* D_Y |_F)$ is log canonical if and only if the pair $(\tilde{Y}, F+ f^* D_Y)$ is log canonical.
 \end{proof}
 
 Recall that we have a stratification $X = \cup_{l\geq 0} X_l$ of $X$ into locally closed subschemes, and the corresponding stratification of the log jet schemes $J_m(\bX) = \cup_l   J_m(\bX)_l$, where each $J_m(\bX)_l$ is a locally trivial $\bfA^{ml}$-bundle over $J_m(X_l)$. By Assumption~\ref{ass-strat}, $X_l$ has codimension $l$ in $X$, hence $(c+l)$ is the codimension of $X_l$ in $Y$.
 
 To simplify notation, we will write $(Y, X_l)$ for the pair $(Y\setmin \cup_{i>l} X_i, X_l)$. 
 
 \begin{lemma} The log jet scheme $J_m(\bX)$ is irreducible for any $m>0$ if and only if the pair $(Y,(c+l)X_l)$ is log canonical for any $l>0$.
 \end{lemma}
 
 \begin{proof}
 Recall that since $X$ is a complete intersection, all components of $J_m(\bX)$ have dimension at least $(m+1)d$, where $d=\dim X$. The log jet scheme $J_m(\bX)$ has no components over $D_X$ if and only if for any $l>0$,
\begin{equation}   \dim J_m(\bX)_l = \dim J_m(X_l) + ml <d(m+1). \label{eq-irred} \end{equation}

By the main theorem in \cite{Mustata2}, the pair $(Y, (c+l)X_l)$ is log canonical if and only if for any $m$
\begin{equation} \dim J_m(X_l) \leq (m+1)(d-l). \label{eq-logc}\end{equation}
Moreover, 
\[ \lct(Y,X_l) = d+c-  \frac{\dim J_m(X_l)}{m+1}\]
for any $m+1$ large and divisible enough.

If $(Y,(c+l) X_l)$ is log canonical for all $l$, then inequality~(\ref{eq-logc}) holds for all $l$. This implies that inequality~(\ref{eq-irred}) also holds for all $l$, hence $J_m(\bX)$ is irreducible. 

Conversely, if  $J_m(\bX)$ is irreducible for all $m$, then the inequality~(\ref{eq-irred}) holds for all $m, l$. Computing the log canonical threshold, we get
\[ \lct(Y,X_l) > d+c - \frac{(m+1)d-ml}{m+1} = c+l -\frac{l}{m+1},\]
for any $m+1$ large and divisible enough and for any $l>0$. This implies that $\lct(Y,X_l) \geq c+l$ and hence $(Y,(c+l) X_l)$ is log canonical for any $l>0$.
\end{proof}

\begin{lemma} The pair $(Y,(c+l)X_l)$ is log canonical if and only if $(Y, cX+D_Y)$ is log canonical near $X_l$.
\end{lemma}

\begin{proof}
Let $x\in X_l$ and let $D_1,\ldots,D_l$ be the irreducible toric divisors in $Y$ containing $x$. Then  
 $X_l = X\cap D_1\cap\cdots \cap D_l$ and $D_Y=D_1+\cdots +D_l$ near $x$. Now apply Lemma~\ref{lem-mustata} to the pair $(Y, cX+D_1+\cdots+D_l)$.
\end{proof}

The last two lemmas show that $J_m(\bX)$ is irreducible if and only if $(Y, cX+D_Y)$ is log canonical near $X_l$ for any $l$. This means that  $J_m(\bX)$ is irreducible if and only if $(Y, cX+D_Y)$ is log canonical. Lemma~\ref{lem-reduc} then finishes the proof of the main theorem in the case when $Y$ is non-singular.

\subsection{The case of arbitrary toric $Y$}

Let $f: \tilde{Y}\to Y$ be a toric resolution of singularities of $Y$. Then $K_{\tilde{Y}}+ D_{\tilde{Y}} = f^*(K_Y+D_Y)$ because both of these divisors are trivial.

Let $\tilde{X} = f^{-1}(X)$. Since $X$ is a complete intersection and it intersects the torus orbits of $Y$ properly, it follows that $\tilde{X}$ is an irreducible, complete intersection in $\tilde{Y}$, and the morphism $f:\tilde{X}\to X$ is birational. Let $\tilde{X}$ have the log str induced by the log str on $\tilde{Y}$. 

\begin{lemma} $J_m(\bX)$ is irreducible if and only if $J_m(\tilde{\bX})$ is irreducible.
\end{lemma}

\begin{proof}
Let $y\in \tilde{Y}$. Since the log morphism $f: \tilde{\bY}\to \bY$ is birational and log smooth, the morphism of jets
\[ J_m(\tilde{\bY})_y\to J_m(\bY)_{f(y)}\]
is an isomorphism. Here $J_m(\tilde{\bY})_y$ is the fibre of $\pi_m: J_m(\tilde{\bY}) \to Y$ over $y$.

A log jet into $\bX$ is the same as a log jet into $\bY$, with image lying in $X$. This implies that if $y\in \tilde{X}$ then  
\[ J_m(\tilde{\bX})_y\to J_m(\bX)_{f(y)}\]
is also an isomorphism.

If $J_m(\tilde{\bX})$ is irreducible, so is $J_m(\bX)$ because it is the image of $J_m(\tilde{\bX})$.

Conversely, suppose $J_m(\tilde{\bX})$ is reducible. Then for some $l$, 
\[ \dim J_m(\tilde{X}_l) \geq (m+1)d-ml.\]
Let $S\subset \tilde{Y}$ be a torus orbit of codimension $l$, such that 
\[ \dim J_m(\tilde{X}\cap S) \geq (m+1)d-ml.\]
Let $T\subset Y$ be the image of $S$. Then $T$ is a torus orbit of codimension $j\geq l$. The morphism $f: S\to T$ has $(j-l)$ dimensional fibres. It follows that
\[ \dim J_m(X\cap T) \geq (m+1)d -ml-(j-l) = (m+1)d-mj +[(m-1)(j-l)].\]
Since $(m-1)(j-l)\geq 0$,
\[ \dim J_m(X_j) \geq (m+1)d-mj,\]
hence $J_m(\bX)$ is reducible.
\end{proof}

\begin{lemma} The pair $(X,D_X)$ is log canonical if and only if the pair $(\tilde{X},D_{\tilde{X}})$ is log canonical.
\end{lemma}

\begin{proof}
The morphism $f:\tilde{X}\to X$ is proper birational, hence it is enough to show that $f^*(K_X+D_X) = K_{\tilde{X}}+D_{\tilde{X}}$.

Since $X$ is a complete intersection in $Y$, it follows that 
\[ K_X+D_X = (K_Y+D_Y)|_X.\]
Similarly,
\begin{eqnarray*}  K_{\tilde{X}}+D_{\tilde{X}} &=& K_{\tilde{Y}}+D_{\tilde{Y}}|_{\tilde{X}} \\
&=& f^*(K_Y+D_Y) |_{\tilde{X}} \\
&=& f^*(K_X+D_X).
\end{eqnarray*}

\end{proof}

The two lemmas reduce the case of general $Y$ to the case of smooth $Y$. Note that $\tilde{X}$ may not be normal, but since $f$ is an isomorphism in codimension $1$, $\tilde{X}$ is non-singular in codimension $1$ and the non-normal points all lie in $D_{\tilde{X}}$.

\bibliographystyle{plain}
\bibliography{log}{}

\end{document}